\documentclass{sutj}
\usepackage[utf8]{inputenc}
\usepackage{amsthm,amsmath,amssymb,physics}
\usepackage{overpic,hyperref,cleveref}
\newtheorem{defi}{Definition}
\crefname{defi}{Definition}{}
\newtheorem{lem}{Lemma}
\crefname{lem}{Lemma}{Lemmas}
\newtheorem{theo}{Theorem}
\crefname{theo}{Theorem}{Theorems}
\newtheorem{rem}{Remark}
\crefname{rem}{Remark}{}
\newcommand{\ep}{\varepsilon}
\newcommand{\la}{\Delta}
\newcommand{\diam}{{\rm diam}}
\newcommand{\dist}{{\rm dist}}

\newcommand{\N}{\mathbb{N}}
\newcommand{\J}{\mathcal{J}}
\newcommand{\R}{\mathbb{R}}
\newcommand{\Z}{\mathbb{Z}}
\newcommand{\duS}[2]{\ev{#1,#2}_{H^{-1}(\Omega)}}
\newcommand{\duL}[2]{\ev{#1,#2}_{L^1(\R^d)^*}}
\newcommand{\dmas}[2]{\int_{#1}#2d\sigma}
\newcommand{\open}[1]{\mathring{#1}}
\newcommand{\cls}[1]{\overline{#1}}
\newcommand{\pd}{\partial}
\newcommand{\SU}{\bigsqcup}
\begin{document}
\title[Poisson equation in domains with concentrated holes]{Poisson equation in domains with concentrated holes}
\author[H. Ishida]{Hiroto Ishida}
\address{Hiroto Ishida\\Graduate School of Science, University of Hyogo\\Shosha, Himeji, Hyogo 671-2201, Japan}
\email{immmrfff@gmail.com}
\date{\today}
\maketitle
\begin{abstract}
We consider solutions $u^\ep$ of Poisson problems with the Dirichlet condition on domains $\Omega_\ep$ with holes concentrated at subsets of a domain $\Omega$ non-periodically. We show $u^\ep$ converges to a solution of a Poisson problem with a simple function potential. This is a generalized result of a sample model given by Cioranescu and Murat (1997). They showed a result for case that holes are distributed at $\Omega$ periodically.
\end{abstract}
\section{Introduction}
Let $\Omega\subset\R^d,d\geq 2$ be open and bounded with $C^2$ boundary. We consider a union $T_\ep$ of holes concentrated at subsets of $\R^d$ as \Cref{domep}, and domains $\Omega_\ep=\Omega\setminus T_\ep.$
We consider Poisson problems on $\Omega_\ep$ with the homogeneous Dirichlet condition with $f\in L^2(\Omega),$ that is,
\begin{equation}\label{solus}
u^\ep\in H_0^1(\Omega_\ep),~~-\la u^\ep=f.
\end{equation}
We will see $u^\ep$ converge to $u$ as $\ep\to 0$ which satisfies
\begin{equation}\label{PDE}
u\in H_0^1(\Omega),~~(-\la+V) u=f,
\end{equation}
where $V$ is a simple function.
Details of assumptions for $T_\ep$ and the main result are given in \Cref{ass}.
\begin{figure}\centering
\begin{overpic}[width=8cm]{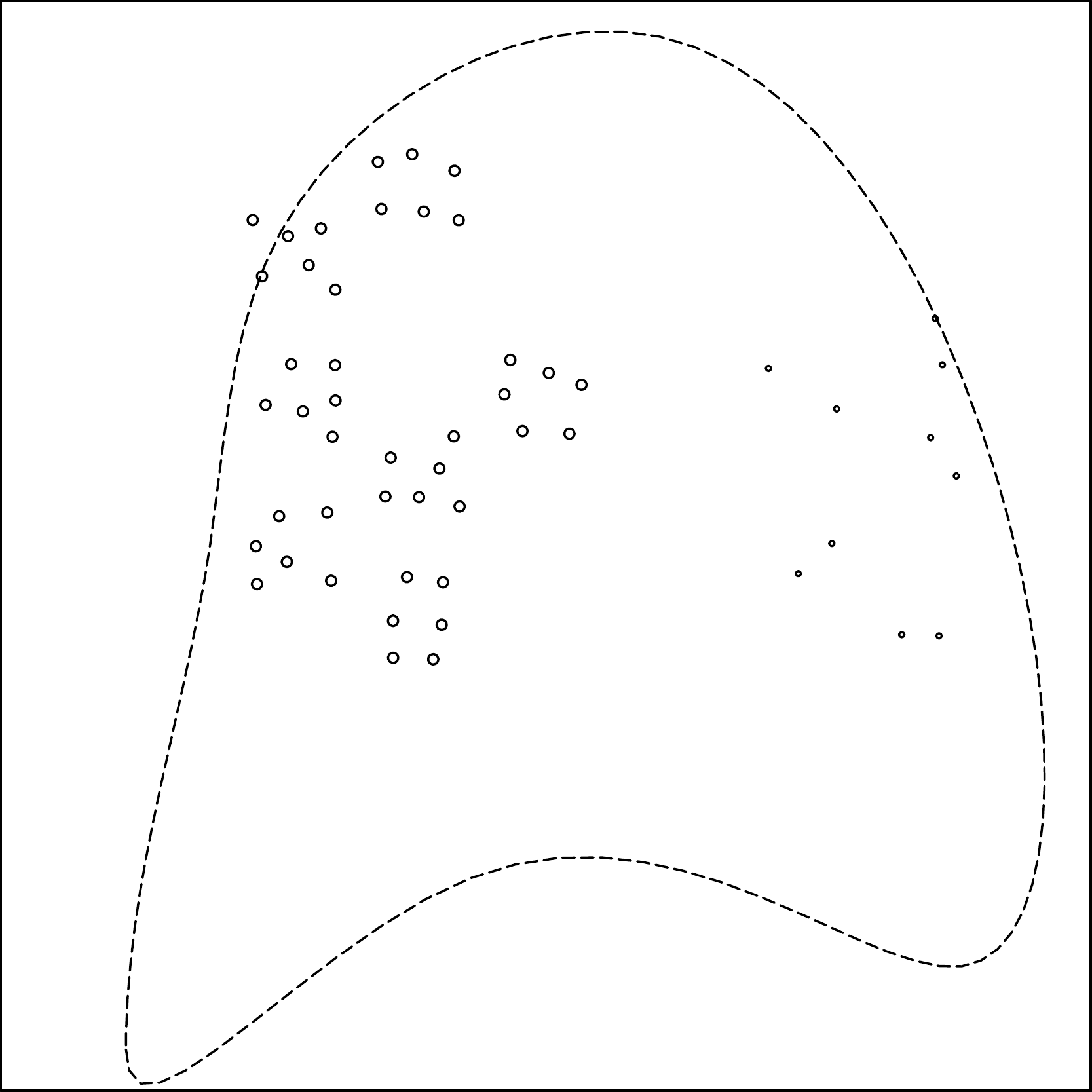}
\put(50,22){$\Omega$}
\put(35,60){$T_\ep$}
\end{overpic}
\caption{A domain $\Omega$ and holes $T_\ep.$\label{domep}}
\end{figure}
\subsection{Known results}
There are many contributions to characterize the limit $u$ of solutions $u^\ep$ on domains $\Omega_\ep$ when $\Omega_\ep\to\Omega$ in a proper sense.
The PDE of the form \eqref{PDE} is often used to characterize the limit $u$.
Many examples with $V=0$ are introduced at \cite{RT}, for example, $\Omega_\ep\to\Omega\setminus K$ metrically with thin $K.$

On the other hand, there are examples for which $V\neq 0.$
The case when $T_\ep=\bigcup_{i\in2\ep\Z^d}\cls{B(i,a_\ep)}$ with the critical radius $a_\ep$ is introduced at \cite[Example 2.1]{CM}, where $a_\ep$ satisfies the same condition for $a_{\ep,k}$ of \eqref{tmu} below. In this case, $V$ is a constant.
A similar result for Robin condition is given by \cite{K} with a different critical radius and a different constant $V$. These results can be regarded as a strong resolvent convergence of Laplacian, and they were improved to a norm resolvent convergence of Laplacian with Dirichlet, Robin and Neumann conditions by \cite{NRC}. In these cases, $V$ is still a constant.
 
Other examples for which $V\neq 0$ are also introduced at \cite[Example 2.9]{CM}. If $T_\ep$ is a union of holes on a hyper plane, $V$ is a Dirac measure supported on the hyper plane.

As for randomly perforated domains,
convergence of solutions in a proper sense with holes whose centers are generated by either Poisson or stationary point process is given by \cite{AG}, \cite{AGst} with a constant $V.$
\section{Assumption and the main result}\label{assresu}
\subsection{Assumption}\label{ass}
We denote Lebesgue measure on $\R^d$ by $|\cdot|.$ We use a class $\J$ of sets to determine where holes concentrate.
\begin{defi}
Let
\[\J=\{E\subset\R^d\mid|\pd E|=0\}.\]
\end{defi}
\begin{rem}
If $E\subset\R^d$ and $|\cls{E}|<\infty,$ $E\in\J$ if and only if $|\cls{E}|=|\open{E}|$ by $\pd{E}=\cls{E}\setminus|\open{E}|.$
Elements of $\J$ are measurable by completeness of Lebesgue measure.
\end{rem}
We shall construct holes $T_\ep$ as follows (see \Cref{sets}). Let $m\in\N$, $\{F_k\}_{k=1}^m\subset\J$ be a collection of disjoint sets and  $\{N_k\}_{k=1}^m\subset\N.$
We use $\SU$ instead of $\bigcup$ for the disjoint union of sets. Let $A\subset\R^d$ be measurable and bounded, and $\Lambda\subset\R^d$ be countable such that \begin{equation}\label{devRd}
\R^d=\SU_{i\in\Lambda}(A+i)~~(A+i=\{x+i\mid x\in A\}).
\end{equation}
For $x\in\R^d$ and $R>0,$ we denote $B(x,R)=\{y\in\R^d\mid|x-y|<R\}.$
 Choose small $C>0$ with
\begin{equation}\label{holesintiles}
|A|>\max_{k\leq m}{N_k}|B(0,C)|.
\end{equation}

We denote $A_i^\ep=\ep(A+i)=\{\ep x\mid x\in A+i\}.$ Remark $\R^d=\SU_{i\in\Lambda}A_i^\ep$ follows from \eqref{devRd} for each $\ep>0.$
\begin{defi}
For $E\subset\R^d$ and $\ep>0,$ let
\[\Lambda^-_\ep(E)
=\{i\in\Lambda\mid A_i^\ep\subset E\},
~~\Lambda^+_\ep(E)
=\{i\in\Lambda\mid A_i^\ep\cap E\neq\emptyset\}.\]
\end{defi}
For $\ep>0$ and $i\in\Lambda^-_\ep(F_k)$ (such $k$ is unique for each $i$), consider centers of holes
$\{x_{i,j}^\ep\mid j=1,...,N_k\}\subset\R^d$
with $\SU_{j=1}^{N_k}B(x_{i,j}^\ep,C\ep)\subset A_i^\ep$ for $\ep\ll 1.$ We omit to write ($\ep\to 0$) for convergence of sequences indexed by $\ep>0.$
Consider radii of holes $a_{\ep,k}$ with the following condition for $1\leq k\leq m$:
\begin{equation}\label{tmu}
\ep^{-d}\cross
\begin{cases}
(-\log a_{\ep,k})^{-1} & (d=2)\\
(a_{\ep,k})^{d-2}& (d\geq3)
\end{cases}
\to\tilde{\mu_k}\in[0,\infty).
\end{equation}
We recall that $\Omega$ is bounded, open with $C^2$ boundary.
We denote
\[
T_{\ep,k}=\SU_{i\in\Lambda^-_\ep(F_k),j\leq N_k}\cls{B(x_{i,j}^\ep,a_{\ep,k})},~~
T_\ep=\SU_{k=1}^m T_{\ep,k},~~
\Omega_\ep=\Omega\setminus
T_\ep.
\]
\begin{figure}[ht]\centering
\begin{overpic}[width=8cm]{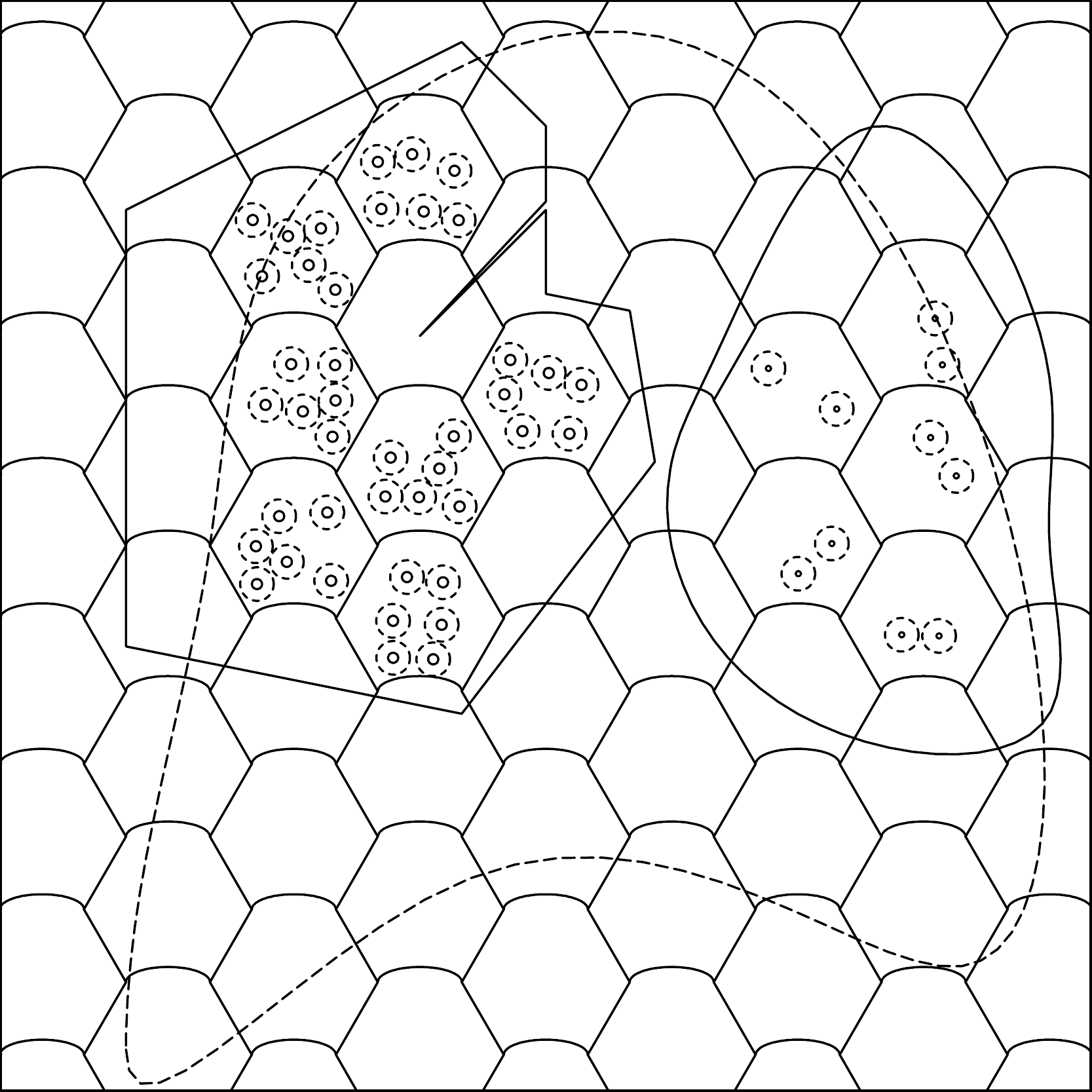}
\put(50,22){$\Omega$}
\put(12,56){$F_1$}
\put(62,56){$F_2$}
\put(59,30){$A_i^\ep$}
\end{overpic}
\caption{Construction of holes $T_\ep$ with $m=2,~N_1=6,~N_2=2.$\label{sets}}
\end{figure}
\subsection{Result}
Using the surface area $S_d$ of $\pd B(0,1),$ we write
$\mu_d=\frac{S_d}{|A|}\cross
\begin{cases}
1&(d=2)\\
d-2&(d\geq3)
\end{cases}.$
For $E\subset\R^d,$ we denote $1_E(x)=\begin{cases}
1 & (x\in E)\\
0 & (x\notin E)
\end{cases}.$
Our main result is stated as follows.
\begin{theo}\label{resu}
Under the assumptions as in \Cref{ass},
$u^\ep$ in \eqref{solus} converges to 
$u$ weakly in $H_0^1(\Omega)$ and the limit $u$ solves \eqref{PDE} with
\[V=\mu_d\sum_{k=1}^m\tilde{\mu_k}N_k1_{F_k}.\]
\end{theo}
\begin{rem}
\cite[Example 2.1]{CM} is just \Cref{resu} with $F_1=\R^d,A=[-1,1)^d,\Lambda=2\Z^d,N_1=1,x_{i,1}^\ep=i\ep.$ It means holes are distributed on $\Omega$ periodically. We generalized it for the case where holes distributed concentrated at $F_k$ non-periodically. Moreover, each $F_k$ can have different density $\tilde{\mu_k}N_k$.
\end{rem}
\subsection{Outline of proof}
The proof of our main result is based on the theorem below.
\begin{theo}[{\cite[Theorem 1.2]{CM}}]\label{CM}
Assume that $T_\ep\subset\R^d$ is closed for each $\ep>0.$ Assume there is a sequence
\begin{equation}\tag{H.1}\label{H1}
\{w^\ep\}\subset H^1(\Omega)
\end{equation}
satisfying
\begin{equation}\tag{H.2}\label{H2}
w^\ep=0\mbox{ on }T_\ep\mbox{ for each }\ep>0,
\end{equation}
\begin{equation}\tag{H.3}\label{H3}
w^\ep\to 1\mbox{ weakly in }H^1(\Omega),
\end{equation}
and there is
\begin{equation}\tag{H.4}\label{H4}
V\in W^{-1,\infty}(\Omega)
\end{equation}
(thus, $V\in H^{-1}(\Omega)$) such that
\begin{equation}\tag{H.5}\label{H5}
\begin{aligned}
&\duS{-\la w^\ep}{\varphi v^\ep}\to\duS{V}{\varphi v}\\
&\mbox{ if }\varphi\in C_0^\infty(\Omega),~
v^\ep=0\mbox{ on }T_\ep.~v^\ep\to v \mbox{ weakly in }H^1(\Omega).
\end{aligned}
\end{equation}
Then, $u^\ep$ in \eqref{solus} converges to 
$u\in H_0^1(\Omega)$ weakly in $H_0^1(\Omega)$ where $u$ is solution to \eqref{PDE}.
\end{theo}
We check the conditions \eqref{H1}--\eqref{H5}  to prove \Cref{resu}. As mentioned in \cite{CM}, it is not unusual that assuming the condition \eqref{H5}.\medskip

We first prepare some lemmas in \Cref{appset}, and we introduce $w^\ep$ and verify the conditions \eqref{H1}--\eqref{H4} in \Cref{corr}. Finally, we check the condition \eqref{H5} in \Cref{end} and complete the proof of  \Cref{resu}.
\section{Proof}\label{proof}
\subsection{Approximation of sets by tiles \texorpdfstring{$A_i^\ep$}{}}\label{appset}
We first state some properties for $\J.$
\begin{lem}\label{capbou}
Let $E_1,E_2\in\J,$ then $|\cls{E_1\cap E_2}|=|(E_1\cap E_2)^\circ|.$
\end{lem}
\begin{proof}
A distributive property for sets shows
$\cls{E_1\cap E_2}
\subset\cls{E_1}\cap\cls{E_2}
=(\open{E_1}\sqcup\pd{E_1})\cap(\open{E_2}\sqcup\pd{E_2})
=(\open{E_1}\cap\open{E_2})\cup E
=(E_1\cap E_2)^\circ\cup E
$ with some $E$ satisfying $|E|=0.$
\end{proof}
\begin{defi}
For $E\subset\R^d$ and $\ep>0,$ let
\[
A_\ep^\pm(E)=\SU_{i\in\Lambda^\pm_\ep(E)}A_i^\ep.
\]
\end{defi}
We remark that $A^-_\ep(E)\subset E\subset A^+_\ep(E).$ We will see that they are approximations of $E$  by \Cref{tilesup,tileinf} below.
\begin{lem}\label{tilesup}
Let $E\subset\R^d$ be measurable and bounded, and satisfy $|E|=|\cls{E}|.$ Then $|A_\ep^+(E)|$ $\to |E|.$
\end{lem}
\begin{proof}
Let $d_\ep=\diam(\ep A).$ Then  $d_\ep\to 0.$ 
Let $E_\ep=\bigcup_{x\in E}\cls{B(x,d_\ep)}.$ Then  $\bigcap_{\ep>0} E_\ep=\cls E$ and $|E_\ep|<\infty.$ Thus $|E_\ep|\to|\cls{E}|=|E|.$ The assertion follows from it and $E_\ep\supset A_\ep^+(E)\supset E.$
\end{proof}
\begin{lem}\label{tileinf}
Let $E\subset\R^d$ be a measurable set such that  $|\open{E}|=|E|.$ Then $|A_\ep^-(E)|\to |E|.$
\end{lem}
\begin{proof}
Let $V=\open{E},g(x)=\dist(x,\pd V),d_\ep=\diam(\ep A)$ and $V_{-\ep}=V\cap g^{-1}((d_\ep,\infty)).$
Then $\bigcup_{\ep>0} V_{-\ep}=V$ since $V$ is open.
The assertion follows from $V_{-\ep}\subset A_\ep^-(V)\subset E.$ We verify $V_{-\ep}\subset A_\ep^-(V).$
Let $x\in V_{-\ep}.$ There is $i\in\Lambda$ that $x\in A_i^\ep$. We show $i\in\Lambda^-_\ep(V).$ It is equivalence with $\R^d\setminus V\subset \R^d\setminus A_i^\ep.$ If $y\notin V,$ we can get $p\in\pd V$ from line segment which contain $\{x,y\}$. It is $p_t=(1-t) x+t y$ with minimal $t\in[0,1]$ that $p_t\notin V.$ Construction of $p$ imply $|x-y|=|x-p|+|p-y|\geq\dist(x,\pd V)>d_\ep.$ Thus $y\notin A_i^\ep.$ Thus $i\in\Lambda^-_\ep(V).$
\end{proof}
We can count how many tiles $A^\pm_\ep(E)$ has.
\begin{lem}\label{counttiles}
For $E\subset\R^d$ and $\ep>0,$ the number of elements of $\Lambda^\pm_\ep(E)$ is $\frac{|A^\pm_\ep(E)|}{\ep^d|A|}.$
\end{lem}
We say $E$ is a cube if $E=[0,R)^d+x$ with some $x\in\R^d,R>0.$ We prepare lemmas related to weak star topology of $L^\infty(\R^d)=L^1(\R^d)^*.$
We denote $\duL{g}{h}=\int gh dx$ for $g\in L^\infty(\R^d)=L^1(\R^d)^*,h\in L^1(\R^d).$
\begin{lem}\label{denseL1}
Let $\{g_\ep\}\subset L^\infty(\R^d)$ be bounded and $g\in L^\infty(\R^d).$ If \[\duL{g_\ep}{1_E}\to\duL{g}{1_E}\]
for any cube $E$, $g_\ep\to g$ weakly star in $L^\infty(\R^d).$ 
\end{lem}
\begin{proof}
If follows from the fact that the vector space generated by $\{1_E|E:\mbox{cube}\}$ is dense at $L^1(\R^d)$.
And the fact follows from the facts that the set of simple functions on $\R^d$ is dense in $L^1(\R^d),$ the Lebesgue measure is outer regular and any open set can be represented as the union of disjoint countable cubes.
\end{proof}
\begin{lem}\label{prodws}
If $f_\ep\to f$ in $L^2(\R^d),|f_\ep|\leq 1$ for $\ep\ll 1$ and $g_\ep\to g$ weakly star in $L^\infty(\R^d),$
we have $f_\ep g_\ep\to fg$ weakly star in $L^\infty(\R^d).$
\end{lem}
\begin{proof}
The existence of a subsequence of $f_\ep$ converging to $f$ a.e. gives $|f|\leq 1$ a.e. 
The assertion follows from $c:=\sup_{\ep>0}\norm{g_\ep}_{L^\infty(\R^d)}<\infty,$ \Cref{denseL1}, and
$|\ev{f_\ep g_\ep-fg,1_E}|
\leq c\norm{f_\ep-f}_{L^2(\R^d)}\norm{1_E}_{L^2(\R^d)}+|\ev{g_\ep-g,f1_E}|$ for any cube $E.$
\end{proof}
\subsection{Error corrector \texorpdfstring{$w^\ep$}{}}\label{corr}
By \eqref{tmu}, we have $\frac{\max_k a_{\ep,k}}{\ep}\to0.$ Thus $\displaystyle\max_{k\leq m} a_{\ep,k}$ $<C\ep$ for $\ep\ll 1$ (recall $C>0$ is chosen to satisfy \eqref{holesintiles}).
\[
w_{0,k}^\ep(r)
=\begin{cases}
\frac{\log a_{\ep,k}-\log r}{\log a_{\ep,k}-\log C\ep}&(d=2)\\
\frac{(a_{\ep,k})^{-d+2}-r^{-d+2}}{(a_{\ep,k})^{-d+2}-(C\ep)^{-d+2}}&(d\geq3)
\end{cases}~(a_{\ep,k}\leq r\leq C\ep),
\]
\[
B_{\ep,k}=\SU_{i\in\Lambda^-_\ep(F_k),j\leq N_k}B(x_{i,j}^\ep,C\ep),~~
B_\ep=\SU_{k=1}^m B_{\ep,k},\]
\[
w^\ep(x)=\begin{cases}
0&(x\in T_\ep)\\
w_{0,k}^\ep(|x-x_{i,j}^\ep|)&(x\in B(x_{i,j}^\ep,C\ep)\setminus B(x_{i,j}^\ep,a_{\ep,k}))\\
1 & (x\notin B_\ep)
\end{cases}.
\]
Then we have
\begin{equation}\label{harmwep}
\la w^\ep=0~{\rm on}~B_\ep\setminus T_\ep.
\end{equation}
and \eqref{H2}.
We need the limit of $1_{B_{\ep,k}}$ to analyze $w^\ep.$
\begin{lem}\label{bep}
$1_{B_{\ep,k}}\to\frac{N_k|B(0,C)|}{|A|}1_{F_k}=\frac{N_k C^d S_d}{d|A|}1_{F_k}$ weakly star in $L^\infty(\R^d).$
\begin{proof}
Let $E$ be a cube.
By $|B_{\ep,k}\cap A_i^\ep|=\begin{cases}
N_k|B(0,C\ep)|&(i\in\Lambda^-_\ep(F_k))\\
0&(i\notin\Lambda^-_\ep(F_k))
\end{cases},$ \Cref{counttiles} and $B_{\ep,k}\subset F_k,$ we have
\begin{align*}
\frac{|A^-_\ep(E\cap F_k)|}{\ep^d| A|}N_k|B(0,C\ep)|
&=|B_{\ep,k}\cap A^-_\ep(E\cap F_k)|
\leq\duL{1_{B_{\ep,k}}}{1_E}\\
&\leq\frac{|A^+_\ep(E\cap F_k)|}{\ep^d| A|}N_k|B(0,C\ep)|.
\end{align*}
By \Cref{capbou,tilesup,tileinf},
\[\frac{|A^-_\ep(E\cap F_k)|}{\ep^d| A|}N_k|B(0,C\ep)|
\to\frac{|E\cap F_k|N_k|B(0,C)|}{|A|}
=\ev{\frac{N_k|B(0,C)|}{|A|}1_{F_k},1_E}.\]
These, \Cref{denseL1} and $|B(0,C)|=\frac{S_d C^d}{d}$ imply the assertion.
\end{proof}
\end{lem}
\begin{lem}
We have \eqref{H1} and \eqref{H3}
\end{lem}
\begin{proof}
For $i\in\Lambda_\ep^-(F_k),~j\leq N_k,~k\leq m,$ We have 
\begin{align*}\norm{\grad w^\ep}_{L^2(B(x_{i,j}^\ep,C\ep)\setminus\cls{B(x_{i,j}^\ep,a_{\ep,k})})}^2
&=S_d\int_{a_{\ep,k}}^{C\ep}|\pd_r w_{0,k}^\ep(r)|^2r^{d-1}d r\\
&=S_d\begin{cases}
\frac{1}{\log C\ep-\log a_{\ep,k}}&(d=2)\\
\frac{d-2}{(a_{\ep,k})^{-d+2}-(C\ep)^{-d+2}}&(d\geq3)
\end{cases},
\end{align*}
which along with $|w^\ep|\leq 1$ implies $w^\ep$ is an extension of an $H_{loc}^1(B_\ep\setminus T_\ep)$ function by the boundary values on $\pd(B_\ep\setminus T_\ep).$ Thus, $\grad w^\ep$ in the distributional sense coincides with the pointwise, classical derivative and
\[\norm{\grad w^\ep}_{L^2(A_i^\ep)}^2
=\begin{cases}
\frac{N_kS_d}{\log C\ep-\log a_{\ep,k}}&(i\in\Lambda^-_\ep(F_k),~d=2)\\
\frac{N_kS_d(d-2)}{(a_{\ep,k})^{-d+2}-(C\ep)^{-d+2}}&(i\in\Lambda^-_\ep(F_k),~d\geq3)\\
0&(i\notin\bigcup_{k\leq m}\Lambda^-_\ep(F_k))
\end{cases}.\]
Using \eqref{tmu} for them, we have $ c:=\sup_{\ep>0,i\in\Lambda}\ep^{-d}\norm{\grad w^\ep}_{L^2(A_i^\ep)}^2<\infty.$
Thus $\norm{\grad w^\ep}_{L^2(A_i^\ep)}^2\leq c\ep^d.$
It and \Cref{tilesup,counttiles} imply
\[\norm{\grad w^\ep}_{L^2(\Omega)}^2
\leq\norm{\grad w^\ep}_{L^2(A^+_\ep(\Omega))}^2
\leq\frac{|A_\ep^+(\Omega)|}{\ep^d|A|}c\ep^d
\leq\frac{c|\bigcup_{x\in\Omega}B(x,1)|}{|A|}~(\ep\ll 1),
\]
which together with $|w^\ep|\leq 1$ implies \eqref{H1}, and $\{w^\ep\}\subset H^1(\Omega)$ is bounded.

Consider any subsequences of $\{w^\ep\}$ (we still denote $w^\ep$) which converge weakly in $H^1(\Omega),$ and let $w=\mbox{w-}\lim_{\ep\to0}w^\ep.$ We show $w=1.$
Let $F=\sqcup_k F_k.$
Rellich's theorem gives $w^\ep1_{\R^d\setminus F}=1_{\R^d\setminus F}$ tend to $w1_{\R^d\setminus F}=1_{\R^d\setminus F}$ in $L^2(\Omega).$ Thus, $w=1$ a.e. on $\Omega\setminus F.$ On the other hand, \Cref{bep} gives $1_{F_k\setminus B_{\ep,k}}=1_{F_k}(1-1_{B_{\ep,k}})\to 1_{F_k}(1-c_k1_{F_k})=(1-c_k)1_{F_k}$ weakly star in $L^\infty(\R^d)$ where $c_k=\frac{N_k|B(0,C)|}{|A|}.$ Hence 
$w^\ep1_\Omega1_{F_k\setminus B_{\ep,k}}=1_\Omega1_{F_k\setminus B_{\ep,k}}$ tends to $w1_\Omega(1-c_k)1_{F_k}=1_\Omega(1-c_k)1_{F_k}$ weakly star in $L^\infty(\R^d)$ for each $k$ by \Cref{prodws}.
Since $0<c_k<1$ by \eqref{holesintiles}, we have $w=1$ on $\Omega\cap F_k.$ Since $\R^d=(\R^d\setminus F)\cup(\sqcup_k F_k),$ we have $w=1$ on $\Omega.$
\end{proof}
We use a special function to analyze a distribution $-\la w^\ep.$ Let
\[q_0^\ep(r)
=\frac{r^2-(C\ep)^2}{2}~(0\leq r\leq C\ep),
\]
\[
q^\ep(x)
=\begin{cases}
q_0^\ep(|x-x_{i,j}^\ep|)&(x\in B(x_{i,j}^\ep,C\ep))\\
0&(x\notin B_\ep)
\end{cases}.
\]
Then we have
\begin{equation}\label{qepder}
-\la q^\ep=-d~(x\in B_\ep),~
\pd_r q_0^\ep(C\ep)
=C\ep,~
q_0^\ep(C\ep)=0.
\end{equation}
Now we decompose the restricted distribution $(-\la w^\ep)|_{H_0^1(\Omega_\ep)}$ by using $q^\ep.$
\begin{lem}\label{lawep}
Suppose $v\in H_0^1(\Omega_\ep).$ Then we have
\[
\duS{-\la w^\ep}{v}=\sum_{k\leq m}\frac{\pd_r w_{0,k}^\ep(C\ep)}{C\ep}\qty(\int_{B_{\ep,k}}\grad q^\ep\cdot\grad v d x+d\duS{1_{B_{\ep,k}}}{v}).
\]
\end{lem}
\begin{proof}
By \eqref{qepder} and integration by parts,
\[
\int_{B_{\ep,k}}\grad q^\ep\cdot\grad v d x
=C\ep\dmas{\pd B_{\ep,k}}{v}-d\duS{1_{B_{\ep,k}}}{v}\]
for $v\in H_0^1(\Omega_\ep).$
By assumption, $\dmas{\pd T_{\ep,k}}{v}=0.$ Using them and \eqref{harmwep}, we have
\begin{align*}
\duS{-\la w^\ep}{v}
&=\sum_{k\leq m}\int_{B_{\ep,k}\setminus T_{\ep,k}}\grad w^\ep\cdot\grad v d x
=\sum_{k\leq m}\pd_r w_{0,k}^\ep(C\ep)\dmas{\pd B_{\ep,k}}{v}\\
&=\sum_{k\leq m}\frac{\pd_r w_{0,k}^\ep(C\ep)}{C\ep}\qty(\int_{B_{\ep,k}}\grad q^\ep\cdot\grad v d x+d\duS{1_{B_{\ep,k}}}{v}).
\end{align*}
This completes the proof.
\end{proof}
The following lemma is very similar to \eqref{H5}.
\begin{lem}\label{H5p}
Suppose that $v^\ep\in H_0^1(\Omega_\ep)$ and $v^\ep\to v$ weakly in $H_0^1(\Omega),$ Then 
\[
\duS{-\la w^\ep}{v^\ep}\to\duS{\mu_d\sum_{k=1}^m\tilde{\mu_k}N_k1_{F_k}}{v}.\]
\end{lem}
\begin{proof}
By \eqref{tmu}, we have
$\displaystyle\frac{\pd_r w_{0,k}^\ep(C\ep)}{C\ep}\to\frac{\tilde{\mu_k}}{C^d}\cross
\begin{cases}
1&(d=2)\\
d-2&(d\geq3)
\end{cases}.$
We also have
\[\abs{\int_{B_{\ep,k}}\grad q^\ep\cdot\grad v^\ep d x}
\leq C\ep\sup_{\delta>0}\norm{v^\delta}_{W^{1,1}(\Omega)}\to 0.\]
Rellich's theorem gives $|\duS{1_{B_{\ep,k}}}{v^\ep-v}|
\leq\norm{1}_{L^2(\Omega)}\norm{v^\ep-v}_{L^2(\Omega)}\to 0.$
It and \Cref{bep} give
$\duS{1_{B_{\ep,k}}}{v^\ep}
=\duS{1_{B_{\ep,k}}}{v^\ep-v}$
$+\duL{1_{B_{\ep,k}}}{1_\Omega v}$
$\to\duS{\frac{N_k C^d S_d}{d|A|}1_{F_k}}{v}.$
The assertion follows from these limit and \Cref{lawep}.
\end{proof}
\subsection{Proof of \texorpdfstring{\Cref{resu}}{}}\label{end}
\begin{proof}
Since $V=\mu_d\sum_{k=1}^m\tilde{\mu_k}N_k 1_{F_k}\in L^\infty(\Omega)=L^1(\Omega)^*\subset W^{-1,\infty}(\Omega),$ we have \eqref{H4}.
We shall verify \eqref{H5}.
Indeed, the multiplier of $\varphi:H^1(\Omega)\to H_0^1(\Omega)$ is a bounded operator.
Thus, $\varphi v^\ep\to \varphi v$ weakly in $H_0^1(\Omega).$ It and \Cref{H5p} imply \eqref{H5}.
Since we already checked \eqref{H1}--\eqref{H3} in \Cref{corr}, \Cref{resu} follows from \Cref{CM}.
\end{proof}
\subsection*{Acknowledgement.}
The author thanks to the referees for their suggestions in the improvement of the paper.
\bibliographystyle{plain}
\bibliography{Cite/CM,Cite/K,Cite/NRC,Cite/RT,Cite/AG,Cite/AGst}

\end{document}